\documentclass[twoside,a4paper,12pt,reqno]{amsart}
\usepackage{fullpage}
\usepackage{amsfonts}
\usepackage{amssymb}
\usepackage{amsmath}
\usepackage{mathrsfs}
\usepackage{enumitem}
\usepackage{tikz}
\usepackage{float}
\usepackage{subcaption}
\usepackage{cite}
\makeatletter
\newcommand{\citecomment}[2][]{\citen{#2}#1\citevar}
\newcommand{\citeone}[1]{\citecomment{#1}}
\newcommand{\citetwo}[2][]{\citecomment[,~#1]{#2}}
\newcommand{\citevar}{\@ifnextchar\bgroup{;~\citeone}{\@ifnextchar[{;~\citetwo}{]}}}
\newcommand{\citefirst}{\@ifnextchar\bgroup{\citeone}{\@ifnextchar[{\citetwo}{]}}}
\newcommand{\cites}{[\citefirst}
\makeatother
\theoremstyle{plain}
\newtheorem*{claim*}{Claim}
\newtheorem{thm}{Theorem}[section]
\newtheorem{cor}[thm]{Corollary}
\newtheorem{lem}[thm]{Lemma}
\newtheorem{prop}[thm]{Proposition}
\theoremstyle{definition}
\newtheorem{defn}[thm]{Definition}
\newtheorem{ex}[thm]{Example}

\newtheorem{rem}[thm]{Remark}
\newtheorem{con}[thm]{Construction}

\newtheorem{op}[thm]{Open Problem}

\renewcommand{\l}{\mbox{${\mathcal{L}}$}}
\renewcommand{\r}{\mbox{${\mathcal{R}}$}}
\newcommand{\h}{\mbox{${\mathcal{H}}$}}
\renewcommand{\d}{\mbox{${\mathcal{D}}$}}
\renewcommand{\j}{\mbox{${\mathcal{J}}$}}
\renewcommand{\k}{\mbox{${\mathcal{K}}$}}
\newcommand{\leql}{\leq_{\mathcal{L}}}
\newcommand{\geql}{\geq_{\mathcal{L}}}
\newcommand{\leqr}{\leq_{\mathcal{R}}}
\newcommand{\geqr}{\geq_{\mathcal{R}}}

\newcommand{\leqj}{\leq_{\mathcal{J}}}
\newcommand{\geqj}{\geq_{\mathcal{J}}}
\newcommand{\leqk}{\leq_{\mathcal{K}}}
\newcommand{\geqk}{\geq_{\mathcal{K}}}

\newcommand{\gl}{>_{\mathcal{L}}}

\newcommand{\gk}{>_{\mathcal{K}}}

\newcommand{\s}{\sigma}
\renewcommand{\th}{\theta}
\renewcommand{\phi}{\varphi}
\newcommand{\N}{\mathbb{N}}
\newcommand{\Z}{\mathbb{Z}}
\newcommand{\U}{\mathcal{U}}
\renewcommand{\th}{\theta}
\newcommand{\sm}{\setminus}
\newcommand{\mt}{\mapsto}
\newcommand{\ol}{\overline}

\begin{document}
\title{\large{On minimal conditions in semigroups and biacts}}
\author{Craig Miller}
\address{Department of Mathematics, University of York, UK, YO10 5DD}
\email{craig.miller@york.ac.uk}
\maketitle
\begin{abstract}
We systematically study the minimal conditions on $\l$-, $\r$- and $\j$-classes, denoted by $M_L$, $M_R$ and $M_J$, as well as the related notions of left/right/two-sided stability, in semigroups and biacts.  In particular, we investigate the behaviour of these conditions with respect to quotients, substructures and extensions.  Among other results, the following are proved.  The conditions $M_L$, $M_R$ and $M_J$ are preserved under quotients (for both semigroups and biacts), but this is not the case for one- and two-sided stability.  For semigroups, the conditions $M_L$, $M_R$ and one- and two-sided stability are inherited by subsemigroups of finite Green index and also by bi-ideals (and hence by one- and two-sided ideals).  Moreover, a semigroup satisfies $M_L$ (or $M_R$) if and only if both an ideal and the associated Rees quotient do, but the analogue of this result fails in both directions for the condition $M_J$, and in the reverse direction for one- and two-sided stability.
\end{abstract}
~
\maketitle

\section{Introduction}\label{sec:intro}

Just as chain conditions have been instrumental in the development of the structure theory of rings, they have also played a significant role in semigroup theory.  Of particular importance are the minimal conditions on principal left, right and two-sided ideals, introduced by Green in the seminal paper \cite{Green} and denoted by $M_L$, $M_R$ and $M_J$.
These minimal conditions may be formulated in terms of the posets arising from Green's relations, arguably the most important tools for analysing the structure of semigroups.  Specifically, Green's $\l$-preorder on a semigroup $S$ is defined by $a\leql b\Leftrightarrow S^1a\subseteq S^1b,$ and this preorder induces Green's equivalence $\l$ and a partial order on the set of $\l$-classes.  The minimal condition $M_L$ is then equivalent to the minimal condition on the poset of $\l$-classes.  Similarly, the conditions $M_R$ and $M_J$ are equivalent to the minimal conditions on the posets of $\r$- and $\j$-classes, defined in terms of principal right and two-sided ideals, respectively.

In \cite{Munn}, Munn further investigated the conditions $M_L$, $M_R$ and $M_J$, and also introduced the related notions $M_L^*$ and $M_R^*$.  Specifically, a semigroup satisfies $M_L^*$ (resp.\ $M_R^*$) if, for each $\j$-class $J,$ the poset of $\l$-classes (resp.\ $\r$-classes) contained in $J$ satisfies the minimal condition.  A unified account of the results of \cite{Green,Munn} is presented in \cites[Section 6.6]{Clifford:1967}.  Moreover, the article \cite{Hall} systematically investigates the relationships between the conditions $M_L$, $M_R$ $M_J$, $M_L^*$ and $M_R^*$.  The condition $M_L^*$ on a semigroup $S$ is equivalent to $S$ being left stable, meaning that $sa\,\j\,a\Leftrightarrow sa\,\l\,a$ for all $a,s\in S,$ and $M_R^*$ is equivalent to the dual notion of being right stable \cite[Proposition 3.7]{Lallement}.  A semigroup is said to be stable if it is left and right stable\footnote{There are various notions in the literature called ‘stability’ that are closely related but not identical to the modern definition used in this paper; see \cite{East} for a survey.}.  The class of stable semigroups includes all group-bound semigroups \cite[Theorem 1.2(vi)]{Hall} and hence all finite semigroups.  Many of the key properties of finite semigroups in fact hold for stable semigroups, notably that $\j$ coincides with Green's relation $\d:=\l\vee\r$ and that each $\l$- or $\r$-class is minimal within its $\j$-class.  Indeed, it was asserted in \cite[Section A.2]{Rhodes} that stability is `somehow the crucial property that makes finite semigroup theory ``work'''.  
The paper \cite{East} explored the inheritance of Green's relations by subsemigroups in the presence of stability of elements. 

A biact over semigroups $S$ and $T$ is a set $A$ on which $S$ and $T$ act on the left and right, respectively, such that $(sa)t=s(at)$ for all $s\in S,$ $t\in T$ and $a\in A.$  The definitions of Green's relations, and hence of the above minimal conditions, naturally extend to biacts over semigroups.  It turns out that many classical results concerning Green's relations on semigroups, such as Green's Lemma, hold in the more general setting of biacts \cites[p.\! 60-62]{kkm}[Section 3.1]{Mary}.  The notion of stability in biacts was introduced in \cite[Section 3.2]{Mary}.  In particular, it was shown that finite biacts are stable and that stable biacts satisfy the condition $\d=\j.$

Although the conditions $M_L$, $M_R$, $M_J$ and left/right/two-sided stability for semigroups have received significant attention, there has hitherto been little consideration of the behaviour of these conditions with respect to the basic algebraic constructions of quotients, substructures and extensions.  The main purpose of this article is to undertake such an investigation.  The paper is organised as follows.

In Section \ref{sec:prelim} we provide the necessary preliminary material on semigroups and biacts.  In Section \ref{sec:min} we find equivalent formulations for the minimal conditions, and explore the relationships between them.  Sections \ref{sec:M_K} and \ref{sec:stability} contain the main results of the paper.  Section \ref{sec:M_K} is concerned with the conditions $M_L$, $M_R$ and $M_J$, while Section \ref{sec:stability} is concerned with one- and two-sided stability.  In these two sections, we investigate the preservation of the conditions under quotients, substructures and extensions.  The primary focus of this paper will be on semigroups, with biacts serving as an essential auxiliary tool throughout.  Note that we will often state results concerning $M_L$ or left stability but not the implicit counterparts for $M_R$ or right stability.

\section{Preliminaries}\label{sec:prelim}

%


In this section we briefly review the main concepts regarding semigroups and biacts needed for the paper.  We refer the reader to \cite{Clifford:1961} for a comprehensive introduction to semigroup theory, and to \cite[Section 1.4]{kkm} for more information about biacts.

Let $S$ be a semigroup.  Throughout this paper, $S^1$ stands for the monoid obtained from $S$ by adjoining an identity $1\notin S,$ and $S^0$ denotes the semigroup obtained from $S$ by adjoining a zero $0\notin S.$

A {\em left congruence} on $S$ is an equivalence relation $\rho$ on $A$ such that $(sa,sb)\in\rho$ for all $(a,b)\in\rho$ and $s\in S.$  A {\em right congruence} on $S$ is defined dually, and a {\em congruence} on $S$ is a relation that is both a left congruence and a right congruence.

A {\em left ideal} of $S$ a subset $I\subseteq S$ such that $SI\subseteq I.$  {\em Right ideals} are defined dually, and an {\em ideal} of $S$ is a subset that is both a left ideal and a right ideal.  Given an ideal $I$ of $S,$ the {\em Rees quotient of $S$ by} $I,$ denoted by $S/I,$ is the semigroup with underlying set $(S{\sm}I)\sqcup\{0\}$ and multiplication given by
$$a\cdot b=\begin{cases}
ab&\text{if }a,b,ab\in S{\sm}I,\\
0&\text{otherwise.}
\end{cases}$$

Now let $T$ be a semigroup.  An $(S,T)${\em-biact} consists of a set $A$ together with maps 
$$S\times A\to A,\,(s,a)\mt sa\quad\text{ and }\quad A\times T\to A,\,(a,t)\mt at$$ 
such that
$$s(s'a)=(ss')a,\quad(at)t'=a(tt'),\quad(sa)t=s(at)\quad\text{for all }\,a\in A,\,s,s'\in S,\,t,t'\in T.$$

We shall refer to $(S,S)$-biacts simply as $S${\em-biacts}.  

Any ideal $I$ of $S$ may be regarded as an $S$-biact, where the actions are given by left and right multiplication in $S$; we denote this biact by ${_S}I_S$.  


%
Let $A$ be an $(S,T)$-biact.  A {\em subact} of $A$ is an $(S,T)$-biact $B\subseteq A$ such that $S^1BT^1\subseteq B.$  
Note that the subacts of the $S$-biact ${_S}S_S$ are precisely the $S$-biacts ${_S}I_S$.

A {\em left congruence} on $A$ is an equivalence relation $\rho$ on $A$ such that $(sa,sb)\in\rho$ for all $(a,b)\in\rho$ and $s\in S.$  A {\em right congruence} on $A$ is defined dually, and a {\em congruence} on $A$ is a relation that is both a left congruence and a right congruence.
Note that the left/right/two-sided congruences on the $S$-biact ${_S}S_S$ are precisely the left/right/two-sided congruences on the semigroup $S.$  

Given a congruence $\rho$ on $A,$ the {\em quotient} of $A$ by $\rho,$ denoted by $A/\rho,$ is the $(S,T)$-biact whose universe is the set $\{[a]_{\rho} : a\in A\}$ of equivalence classes of $\rho$ and whose actions are given by $s[a]_{\rho}=[sa]_{\rho}$ and $[a]_{\rho}t=[at]_{\rho}$ for all $a\in A,$ $s\in S$ and $t\in T.$  

Given a subact $B$ of $A,$ the {\em Rees quotient of $A$ by} $B,$ denoted by $A/B,$ is the $(S,T)$-biact with underlying set $(A{\sm}B)\sqcup\{0\}$ and actions given by
$$s\cdot a=\begin{cases}
sa&\text{if }a,sa\in A{\sm}B,\\
0&\text{otherwise,}
\end{cases}
\qquad
a\cdot t=\begin{cases}
at&\text{if }a,at\in A{\sm}B,\\
0&\text{otherwise.}
\end{cases}$$
We note that $A/B$ is isomorphic to the quotient of $A$ by the congruence with classes $\{a\}$ ($a\in A{\sm}B$) and $B.$

%

The well-known concepts of Green's preorders and equivalences on semigroups extend naturally to biacts.  We define preorders $\leql,$ $\leqr$ and $\leqj$ on $A$ by
$$a\leql b\Leftrightarrow S^1a\subseteq S^1b,\qquad a\leqr b\Leftrightarrow aT^1\subseteq bT^1,\qquad a\leqj b\Leftrightarrow S^1aT^1\subseteq S^1bT^1.$$
These preorders yield equivalence relations
$$\l=\,\leql\cap\geql,\qquad\r=\,\leqr\cap\geqr,\qquad\j=\,\leqj\cap\geqj$$
on $A.$
Letting $\k$ stand for any of $\l,$ $\r$ and $\j,$ the preorder $\leqk$ induces a partial order on the set of $\k$-classes: 
$$\hspace{12em}K_a\leq K_b\,\Leftrightarrow\,a\leqk b\qquad(\text{where $K_x$ is the $\k$-class of $x\in A$}).$$
Clearly the following inclusions hold: 
$$\leql\,\subseteq\,\leqj,\quad\leqr\,\subseteq\,\leqj,\quad\l\subseteq\j,\quad\r\subseteq\j.$$
Finally, we define $\h=\l\cap\r$ and $\d=\l\vee\r.$  It turns out that $\l$ and $\r$ commute, and hence $\d=\l\circ\r\,(=\r\circ\l)\subseteq\j$ \cite[Remark 1.4.49]{kkm}.  

Note that $\l$ is a right congruence on $A$ and that $\r$ is a left congruence on $A.$  Moreover, if $s\,\l\,s'$ in $S$ then $sa\,\l\,s'a$ in $A$ for all $a\in A,$ and if $t\,\r\,t'$ in $T$ then $at\,\r\,at'$ in $A$ for all $a\in A.$  These facts will be used implicity throughout the paper.

If $A={_S}S_{S},$ then the above preorders and equivalences coincide with the usual Green's preorders and equivalences on the semigroup $S.$  The semigroup $S$ is said to be {\em simple} if it has a single $\j$-class, and {\em bisimple} if it has a single $\d$-class.

\section{Minimal Conditions: Equivalent Formulations and Interconnections}\label{sec:min}

In this section we formally introduce the conditions $M_L$, $M_R$ and $M_J$, along with the concepts of left, right and two-sided stability, for biacts and semigroups.  We then provide equivalent formulations of these conditions and explore their relationships with other related notions.  Many of the results in this section regarding biacts extend known results for semigroups, with similar proofs. 


\begin{defn}
An $(S,T)$-biact {\em satisfies $M_L$} (resp.\ $M_R,$ $M_J$) if its poset of $\l$-classes (resp.\ $\r$-classes, $\j$-classes) satisfies the minimal condition; that is, every non-empty set $\l$-classes (resp.\ $\r$-classes, $\j$-classes) contains a minimal member.

A semigroup $S$ {\em satisfies $M_L$} (resp.\ $M_R,$ $M_J$) if the $S$-biact ${_S}S_S$ satisfies $M_L$ (resp.\ $M_R,$ $M_J$)
\end{defn}

\begin{defn}
An $(S,T)$-biact $A$ is:
\begin{itemize}
\item {\em left stable} if $sa\,\j\,a\Rightarrow sa\,\l\,a$   for all $a\in A$ and $s\in S$;
\item {\em right stable} if $at\,\j\,a\Rightarrow at\,\r\,a$ for all $a\in A$ and $t\in T$; and
\item {\em stable} if it is both left stable and right stable.
\end{itemize}
A semigroup $S$ is {\em left stable} (resp.\ right stable, stable) if the $S$-biact ${_S}S_S$ is left stable (resp.\ right stable, stable).
\end{defn}

The following lemma provides equivalent formulations for the conditions $M_L$, $M_R$ and $M_J$ in terms of descending chains.  This result follows from the more general fact that the minimal condition on a poset is equivalent to the descending chain condition.

\begin{lem}\label{lem:M_K}
Let $\k\in\{\l,\r,\j\},$ and let $M_K$ be the minimal condition associated with $\k.$  The following are equivalent for an $(S,T)$-biact $A$:
\begin{enumerate}[leftmargin=*]
\item[\emph{(1)}] $A$ satisfies $M_K$;
\item[\emph{(2)}] every descending chain of $\k$-classes of $A$ eventually stabilises;
\item[\emph{(3)}] for every descending chain $a_1\geqk a_2\geqk\cdots$ in $A$ there exists some $m\in\N$ such that $a_m\,\k\,a_n$ for all $n\geq m.$
\end{enumerate}
\end{lem}
%

Note that we will use condition (3) of Lemma \ref{lem:M_K} repeatedly throughout the paper without explicit mention.

The next result provides another formulation for the condition $M_L$ for semigroups.

\begin{prop}\label{prop:M_L}
A semigroup $S$ satisfies $M_L$ if and only if, for every semigroup $T,$ every $(S,T)$-biact satisfies $M_L$.
\end{prop}

\begin{proof}
The reverse implication clearly follows immediately from the definition of left stability for semigroups.  For the forward implication, let $A$ be an $(S,T)$-biact, and consider a descending chain $a_0\geql a_1\geql\cdots$ in $A.$  Then there exist $s_i\in S^1$ ($i\in\N$) such that $a_i=s_ia_{i-1}$.  Letting $u_i=s_i\dots s_1$ ($i\in\N$), we have $a_i=u_ia_0$ for each $i\in\N$, and $u_1\geql u_2\geql\cdots$ in $S.$  Since $S$ satisfies $M_L$, there exists $m\in\N$ such that $u_m\,\l\,u_n$ for all $n\geq m.$  It follows that $a_m=u_ma_0\,\l\,u_na_0=a_n$ for all $n\geq m.$  Thus $A$ satisfies $M_L$.
\end{proof}

Next, we provide several equivalent characterisations for a biact to be left stable. 

\begin{prop}\label{prop:lstable}
Let $S$ and $T$ be semigroups.  For an $(S,T)$-biact $A,$ the following are equivalent:
\begin{enumerate}[leftmargin=*]
\item[\emph{(1)}] $A$ is left stable;
\item[\emph{(2)}] $\leql\cap\,\j=\l$ in $A$;
\item[\emph{(3)}] $\leql\cap\,\j\subseteq\l$ in $A$;
\item[\emph{(4)}] $\leql\cap\geqj\,=\l$ in $A$;
\item[\emph{(5)}] $\leql\cap\geqj\,\subseteq\l$ in $A$;
\item[\emph{(6)}] for every $\j$-class $J$ of $A,$ every $\l$-class in the poset of $\l$-classes of $A$ contained in $J$ is minimal;
\item[\emph{(7)}] for every $\j$-class $J$ of $A,$ the poset of $\l$-classes of $A$ contained in $J$ satisfies the minimal condition;
\item[\emph{(8)}] for every $\j$-class $J$ of $A,$ the poset of $\l$-classes of $A$ contained in $J$ has a minimal element.
\end{enumerate}
\end{prop}

\begin{proof}
The equivalence of (1) and (3) follows immediately from the definition of left stability, the equivalence of (2)-(5) is straightforward to show, and it is obvious that (2)$\Rightarrow$(6)$\Rightarrow$(7)$\Rightarrow$(8). 

For (8)$\Rightarrow$(3), consider $a,b\in A$ with $a\leql b$ and $a\,\j\,b.$  By assumption, the poset of $\l$-classes of $A$ contained in the $\j$-class of $b$ has at least one minimal member; let $c\in A$ belong to such an $\l$-class.  Since $a\leql b$ and $b\,\j\,c,$ there exist $s,u\in S^1$ and $v\in T^1$ such that $a=sb$ and $b=ucv.$  Thus $a=sucv.$  Then, letting $d=suc\in S^1c,$ we have $d\leql c$ and $c\,\j\,d.$  It follows by minimality that $c\,\l\,d,$ and hence there exists $s'\in S^1$ such that $c=s'd.$  Then
$$b=ucv=us'dv=us'sucv=us'a\leql a,$$
and hence $a\,\l\,b,$ as required.
\end{proof}

It was shown by X.\! Mary that $\d=\j$ in stable biacts \cite[Lemma 3.15]{Mary}.  The following result characterises stable biacts in terms of the property $\d=\j$ and a certain technical condition.

\begin{prop}
Let $S$ and $T$ be semigroups.  An $(S,T)$-biact $A$ is stable if and only if $\d=\j$ and $\leql\cap\,\r=\l\,\cap\leqr\,=\h$ in $A.$
\end{prop}

\begin{proof}
($\Rightarrow$) For completeness, we prove that $\d=\j.$  Certainly $\d\subseteq\j.$  For the reverse inclusion, consider $a,b\in A$ with $a\,\j\,b.$  Then there exist $s\in S^1$ and $t\in T^1$ such that $a=sbt.$  Clearly $sb\,\j\,b\,\j\,bt$ and $sb\leql b\geqr bt.$  Since $S$ is stable, we have $sb\,\l\,b\,\r\,bt.$  Then, as $\l$ is a right congruence, we have $a=sbt\,\l\,bt.$  Since $\d=\l\circ\r,$ it follows that $a\,\d\,b,$ as required.

Now, using Proposition \ref{prop:lstable} and the fact that $\r\subseteq\j,$ we have
$$\leql\cap\,\r=\,\leql\cap\,(\r\cap\j)=(\leql\cap\,\j)\cap\r=\l\cap\r=\h.$$
A symmetrical argument proves that $\l\,\cap\leqr\,=\h.$

($\Leftarrow$) Consider $a,b\in A$ with $a\leql b$ and $a\,\j\,b.$  Since $\j=\d=\l\circ\r,$ there exists $c\in A$ such that $a\,\l\,c\,\r\,b.$  As $a\leql b,$ we have $c\leql b.$  Since $\leql\cap\,\r=\h,$ it follows that $c\,\h\,b,$ and hence $a\,\l\,b.$  Thus $\leql\cap\,\j\subseteq\l,$ so that $A$ is left stable by Proposition \ref{prop:lstable}.  A dual argument proves that $A$ is right stable, and hence $A$ is stable.
\end{proof}

In the remainder of this section we explore the relationships between the conditions $M_L$, $M_R$, $M_J$ and one- and two-sided stability.  First, we introduce some related notions.

We call an $(S,T)$-biact $A$ {\em$\l$-periodic} if for each $s\in S$ and $a\in A$ there exists some $n\in\N$ such that $s^na\,\l\,s^{n+1}a$.  We say that $S$ is {\em $\l$-periodic} if ${_S}S_S$ is $\l$-periodic.  It is straightforward to show that $S$ is $\l$-periodic if and only if, for every semigroup $T,$ every $(S,T)$-biact is $\l$-periodic.
We dually define $\r$-periodicity of biacts and semigroups.  A semigroup is called {\em group-bound} if each of its elements has some power that belongs to a subgroup.  
It is not hard to show, using Green's Theorem \cite[Theorem 2.16]{Clifford:1961}, that a semigroup is $\l$- and $\r$-periodic if and only if it is group-bound.  This fact can also be deduced from \cite[Theorems 4 and 7]{Drazin}.

\begin{lem}\label{lem:L-periodic}
Let $S$ and $T$ be semigroups and $A$ an $(S,T)$-biact.  If $A$ satisfies $M_L$ then it is $\l$-periodic, and if $A$ is $\l$-periodic then it is left stable.
\end{lem}

\begin{proof}
The first part follows from the fact that, for any $s\in S$ and $a\in A,$ we have $a\geql sa\geql s^2a\geql\cdots$ in $A.$

For the second part, suppose that $A$ is $\l$-periodic, and let $a,b\in A$ be such that $a\leql b$ and $a\,\j\,b.$  Then there exist $s,u\in S^1$ and $v\in T^1$ such that $a=sb$ and $b=uav.$  Then there exists some $n\in\N$ such that $(us)^nb\,\l\,(us)^{n+1}b$; this trivially holds if $s=u=1,$ and follows from $\l$-periodicity otherwise.  Thus, there exists $s'\in S^1$ such that $(us)^nb=s'(us)^{n+1}b.$  Then
$$b=usbv=(us)^nbv^n=s'us(us)^nbv^n
=s'usb=s'ua\leql a,$$
and hence $a\,\l\,b,$ as required.
\end{proof}

From Lemma \ref{lem:L-periodic} and its dual, we deduce the following corollaries.

\begin{cor}\hspace{-0.2em}{\em\cite[Lemma 3.14]{Mary}}\:
For semigroups $S$ and $T,$ every finite $(S,T)$-biact is stable.
\end{cor}

\begin{cor}
Let $S$ and $T$ be semigroups and $A$ an $(S,T)$-biact.  If $S$ is $\l$-periodic, then $A$ is left stable.  Moreover, if $S$ is $\l$-periodic and $T$ is $\r$-periodic, then $A$ is stable.
\end{cor}


For a biact that satisfies $M_J$, satisfying $M_L$ is equivalent to being left stable:

\begin{lem}\label{lem:M_J,M_L}
Let $S$ and $T$ be semigroups, and let $A$ be an $(S,T)$-biact that satisfies $M_J$.  Then $A$ satisfies $M_L$ if and only if it is left stable.
\end{lem}

\begin{proof}
The forward implication follows immediately from Lemma \ref{lem:L-periodic}.  For the converse, consider a descending chain $a_0\geql a_1\geql\cdots$ in $A.$  Then clearly $a_0\geqj a_1\geqj\cdots$.  Since $S$ satisfies $M_J$, there exists $m\in\N$ such that $a_m\,\j\,a_n$ for all $n\geq m.$  Since $S$ is left stable, it follows that $a_m\,\l\,a_n$ for all $n\geq m.$  Thus $A$ satisfies $M_L$.
\end{proof}

We now provide equivalent formulations for a biact to satisfy both $M_L$ and $M_R$.

\begin{prop}\label{prop:M_L+M_R}
Let $S$ and $T$ be semigroups.  For an $(S,T)$-biact $A,$ the following are equivalent:
\begin{enumerate}[leftmargin=*]
\item[\emph{(1)}] $A$ satisfies $M_L$ and $M_R$;
\item[\emph{(2)}] $A$ satisfies $M_J$ and is both $\l$-periodic and $\r$-periodic.
\item[\emph{(3)}] $A$ is stable and satisfies $M_J$.
\end{enumerate}
\end{prop}

\begin{proof}
(1)$\Rightarrow$(2).  By Lemma \ref{lem:L-periodic} and its dual, the $(S,T)$-biact $A$ is $\l$- and $\r$-periodic.  To show that $A$ satisfies $M_J$, consider a descending chain $a_0\geqj a_1\geqj\cdots$ in $A.$  Then there exist $s_i\in S^1$ and $t_i\in T^1$ ($i\in\N$) such that $a_i=s_ia_{i-1}t_i$.  Then $a_i=s_i\dots s_1a_0t_1\dots t_i$ ($i\in\N$).  For $i\geq0,$ let $b_i=s_i\dots s_1a_0$ and $c_i=a_0t_1\dots t_i$, interpreting $b_0=c_0=a_0$.  Then
$$b_0\geql b_1\geql b_2\geql\cdots\qquad\text{and }\qquad c_0\geqr c_1\geqr c_2\geqr\cdots.$$
Since $A$ satisfies $M_L$ and $M_R$, there exists $m\in\N$ such that $b_m\,\l\,b_n$ and $c_m\,\r\,c_n$ for all $n\geq m.$  Therefore, for each $n\geq m$ there exist $x_n\in S^1$ and $y_n\in T^1$ such that $b_m=x_nb_n$ and $c_m=c_ny_n$.  But then 
\begin{align*}
a_m&=s_m\dots s_1a_0t_1\dots t_m=b_mt_1\dots t_m=x_nb_nt_1\dots t_m=x_ns_n\dots s_1a_0t_1\dots t_m\\
&=x_ns_n\dots s_1c_m=x_ns_n\dots s_1c_ny_n=x_ns_n\dots s_1a_0t_1\dots t_ny_n=x_na_ny_n.
\end{align*}
It follows that $a_m\,\j\,a_n$ for all $n\geq m,$ as desired.

(2)$\Rightarrow$(3).  It follows from Lemma \ref{lem:L-periodic} and its dual that $A$ is stable.

(3)$\Rightarrow$(1).  This follows from Lemma \ref{lem:M_J,M_L} and its dual.
\end{proof} 

%


By Propositions \ref{prop:M_L} and \ref{prop:M_L+M_R} (and the dual of the former), we deduce the following corollaries.

\begin{cor}
Let $S$ and $T$ be semigroups.  If $S$ satisfies $M_L$ and $T$ satisfies $M_R$, then every $(S,T)$-biact satisfies both $M_L$ and $M_R$ (and hence is stable and satisfies $M_J$).
\end{cor}

\begin{cor}
For a semigroup $S,$ the following are equivalent:
\begin{enumerate}[leftmargin=*]
\item[\emph{(1)}] $S$ satisfies $M_L$ and $M_R$;
\item[\emph{(2)}] $S$ is group-bound and satisfies $M_J$;
\item[\emph{(3)}] $S$ is stable and satisfies $M_J$;
\item[\emph{(4)}] every $S$-biact satisfies $M_L$ and $M_R$;
\item[\emph{(5)}] every $S$-biact satisfies $M_J$ and is both $\l$-periodic and $\r$-periodic.
\item[\emph{(6)}] every $S$-biact is stable and satisfies $M_J$.
\end{enumerate}
\end{cor}

\begin{proof}
Recalling that $S$ is group-bound if and only if and only if ${_S}S_S$ is $\l$- and $\r$-periodic, it follows from Proposition \ref{prop:M_L+M_R} that (1)-(3) are equivalent.  Moreover, (4)-(6) are equivalent by Proposition \ref{prop:M_L+M_R}.  Finally, (1) and (4) are equivalent by Proposition \ref{prop:M_L} and its dual.
\end{proof}

\vspace{0.1em}
\begin{rem}~
\begin{enumerate}[leftmargin=*]
\item By \cite{Hall}, there exist semigroups that: 
\begin{enumerate}
\item are stable but not group-bound and do not satisfy $M_L$, $M_R$ or $M_J$ (e.g.\ $(\N,+)$);
\item are group-bound but do not satisfy $M_L$, $M_R$ or $M_J$ (e.g.\ a semilattice containing an infinite descending chain);
\item are stable and satisfy $M_L$ but do not satisfy $M_R$ or $M_J$;
\item satisfy $M_L$ and $M_J$ but not $M_R$;
\item satisfy $M_L$ but neither $M_R$ nor $M_J$;
\item satisfy $M_J$ but neither $M_L$ nor $M_R$.
\end{enumerate}
\item There in fact exist semigroups that are bisimple (and hence simple) but satisfy neither $M_L$ nor $M_R$.  An example of such a semigroup is the {\em bicyclic monoid}, defined by the presentation $\langle a,b\,|\,ab=1\rangle$; see \cite[Section 2.1, Exercise 6]{Clifford:1961}.  
\item Let $S$ and $T$ be semigroups, and let $A$ be the $(S,T)$-biact with universe $S\times T$ and actions given by $s(a,b)=(sa,b)$ and $(a,b)t=(a,bt).$  Clearly $(a,b)\leqj(c,d)$ in $A$ if and only if $a\leql c$ in $S$ and $b\leqr d$ in $T.$  It quickly follows that $A$ satisfies $M_J$ if and only if $S$ satisfies $M_L$ and $T$ satisfies $M_R$.
\item It follows from (2) and (3) of this remark that there exist bisimple semigroups $S$ with $S$-biacts that do not satisfy $M_J$.
\end{enumerate}
\end{rem}

\section{The Conditions $M_L$, $M_R$ and $M_J$}\label{sec:M_K}

%

We begin this section by showing that the conditions $M_L$, $M_R$ and $M_J$ on biacts are closed under quotients.

\begin{prop}\label{prop:actquotient}
Let $S$ and $T$ be semigroups, let $A$ be an $(S,T)$-biact, and let $\rho$ be a congruence on $A.$  For $K\in\{L,R,J\},$ if $A$ satisfies $M_K$ then so does $A/\rho.$
\end{prop}

\begin{proof}
We prove the result for $K=J$; the proofs for $L$ and $R$ are similar and slightly more straightforward.

Consider a chain $b_0\geqj b_1\geqj\cdots$ in $A/\rho.$  Then there exist $s_i\in S^1$ and $t_i\in T^1$ ($i\in\N$) such that $b_i=s_ib_{i-1}t_i$.  Choose $a_0\in A$ such that $[a_0]_{\rho}=b_0$, and let $a_i=s_i\dots s_1a_0t_1\dots t_i$.
Then $[a_i]_{\rho}=b_i$ for each $i\in\N,$ and $a_0\geqj a_1\geqj\cdots$ in $A.$  Since $A$ satisfies $M_J$, there exists $m\in\N$ such that $a_m\,\j\,a_n$ for all $n\geq m.$  Therefore, for each $n\geq m$ there exist $u_n\in S^1$ and $v_n\in T^1$ such that $a_m=u_na_nv_n$.  Then $b_m=[a_m]_{\rho}=u_nb_nv_n\in S^1b_nT^1$, so that $b_m\leqj b_n$.  It follows that $b_m\,\j\,b_n$ for all $n\geq m.$  Thus $A/\rho$ satisfies $M_J$.
\end{proof}

We may quickly deduce that the conditions $M_L$, $M_R$ and $M_J$ are also preserved under semigroup quotients.  First, we make the following observation.

\begin{lem}\label{lem:quotient}
Let $S$ be a semigroup and $\rho$ be a congruence on $S.$  For $\k\in\{\l,\r,\j\}$ and $a,b\in S,$ we have $[a]_{\rho}\leqk[b]_{\rho}$ in $S/\rho$ if and only if $[a]_{\rho}\leqk[b]_{\rho}$ in the $S$-biact ${_S}S_S/\rho$.  Consequently, for $K\in\{L,R,J\},$ the semigroup $S/\rho$ satisfies $M_K$ if and only if ${_S}S_S/\rho$ satisfies $M_K$.
\end{lem}

Proposition \ref{prop:actquotient} and Lemma \ref{lem:quotient} together yield:

\begin{cor}\label{cor:sgpquotient}
Let $S$ be a semigroup and $\rho$ a congruence on $S.$  For $K\in\{L,R,J\},$ if $S$ satisfies $M_K$ then so does $S/\rho.$
\end{cor}



For each of the conditions $M_L$, $M_R$ and $M_J$, a biact satisfies the condition if and only if a subact and the associated Rees quotient do:

\begin{prop}\label{prop:actex}
Let $S$ and $T$ be semigroups, let $A$ be an $(S,T)$-biact, and let $B$ a subact of $A.$  For $K\in\{L,R,J\},$ the $(S,T)$-biact $A$ satisfies $M_K$ if and only if both $B$ and $A/B$ satisfy $M_K$.
\end{prop}

\begin{proof}
Let $K\in\{L,R,J\},$ and let $\k$ be the Green's relation associated with $K.$

($\Rightarrow$)  By Proposition \ref{prop:actquotient}, the Rees quotient $A/B$ satisfies $M_K$.  That $B$ satisfies $M_K$ follows from the fact that the $\k$-preorder on $B$ coincides with the restriction to $B\times B$ of the $\k$-preorder on $A.$

($\Leftarrow$)  Consider a descending chain $a_1\geqk a_2\geqk\cdots$ in $A.$  Then either $a_i\in A{\sm}B$ for all $i\in\N,$ in which case $a_1\geqk a_2\geqk\cdots$ in $A/B,$ or else there exists $n\in\N$ such that $a_n\geqk a_{n+1}\geqk\cdots$ in $B.$  In either case, since $A/B$ and $B$ satisfy $M_K$, it follows that the chain $a_1\geqk a_2\geqk\cdots$ in $A$ eventually stabilises, as required.
\end{proof}

The classes of semigroups satisfying the conditions $M_L$, $M_R$ or $M_J$ are not closed under subsemigroups.  For example, the group of integers $\Z$ certainly satisfies each of these conditions, but its subsemigroup $\N$ satisfies none of them.

Consider a subsemigroup $T$ of a semigroup $S.$  We may consider $S$ as a $T$-biact where $T$ acts on $S$ by left and right multiplication in $S$; we denote this biact by ${_T}S_T$.  (There is no ambiguity between the notations ${_T}S_T$ and ${_S}I_S$, where $I$ is an ideal of $S$: in the special case that $S=I=T,$ the biacts ${_S}I_S$ and ${_T}S_T$ coincide.)  We note that the Green's relations on ${_T}S_T$ coincide with the `$T$-relative' Green's relations introduced by Wallace in \cite{Wallace}.  It was argued in \cite[Section 10.4]{Arbib} that these relations are ``potentially powerful tools''.  They eventually found a significant application in the development of the notion of `Green index' for subsemigroups \cite{Cain,Gray}, defined below.

Now, observe that the $T$-biact ${_T}T_T$ is a subact of ${_T}S_T$.  We denote the Rees quotient ${_T}S_T/{_T}T_T$ by ${_T}(S/T)_T$.  By Proposition \ref{prop:actex} we have:

\begin{prop}\label{prop:subsemigroup}
Let $T$ be a subsemigroup of a semigroup $S.$  For $K\in\{L,R,J\},$ the $T$-biact ${_T}S_T$ satisfies $M_K$ if and only if $T$ and ${_T}(S/T)_T$ satisfy $M_K$.
\end{prop}

Note that for each Green's relation $\k$ on ${_T}S_T$, the $\k$-classes lie either entirely in $T$ or entirely in $S{\sm}T.$  The {\em Green index} of $T$ in $S$ is defined to be one more than the number of $\h$-classes of ${_T}S_T$ in $S{\sm}T$; alternatively, it is the number of $\h$-classes of ${_T}(S/T)_T$.  It is easy to see that $T$ has finite Green index in $S$ if and only if ${_T}(S/T)_T$ has finitely many $\l$- and $\r$-classes.


\begin{thm}\label{thm:subsemigroup}
Let $S$ be a semigroup, and let $T$ be a subsemigroup of $S$ such that ${_T}(S/T)_T$ has finitely many $\l$-classes.  Then the following are equivalent:
\begin{enumerate}[leftmargin=*]
\item[\emph{(1)}] $S$ satisfies $M_L$;
\item[\emph{(2)}] $T$ satisfies $M_L$;
\item[\emph{(3)}] ${_T}S_T$ satisfies $M_L$.
\end{enumerate}
\end{thm}

\begin{proof}
Clearly ${_T}(S/T)_T$ satisfies $M_L$, so (2) and (3) are equivalent by Proposition \ref{prop:subsemigroup}.

(1)$\Rightarrow$(2).  Suppose for a contradiction that there exists an infinite strictly descending chain $a_0\gl a_1\gl\cdots$ in $T.$  Then certainly $a_0\geql a_1\geql\cdots$ in $S.$  Since $S$ satisfies $M_L$, there exists $m\in\N$ such that $a_m\,\l\,a_n$ in $S$ for all $n\geq m.$  Let $N$ be the number of $\l$-classes of ${_T}(S/T)_T$.  For each $i\in\{m,\dots,m+N\},$ since $a_i\,\l\,a_{m+N+1}$ in $S,$ there exists $s_i\in S$ such that $a_i=s_ia_{m+N+1}$.  If we had $s_i\in T$ for some $i\in\{m,\dots,m+N\},$ then we would have $a_i\,\l\,a_{m+N+1}$ in $T,$ a contradiction.  Thus $s_i\in S{\sm}T$ for each $i\in\{m,\dots,m+N\}.$  Since there are only $N-1$ $\l$-classes of ${_T}S_T$ contained in $S{\sm}T,$ by the pigeonhole principle there exist $i,j\in\{m,\dots,m+N\}$ with $i<j$ such that $s_i\,\l\,s_j$ in ${_T}S_T$.  Since $\l$ is a right congruence, it follows that $a_i=s_ia_{m+N+1}\,\l\,s_ja_{m+N+1}=a_j$ in ${_T}S_T$.  But then $a_i\in Ta_j$, so that $a_i\,\l\,a_j$ in $T,$ contradicting the fact that $a_i\gl a_j$.  Thus $T$ satisfies $M_L$.

(3)$\Rightarrow$(1).  Suppose for a contradiction that there exists an infinite strictly descending chain $a_0\gl a_1\gl a_2\gl\cdots$ in $S.$  Then, for each $i\in\N,$ there exists $s_i\in S$ such that $a_i=s_ia_{i-1}$.  For each pair $i,j\in\N$ with $i\leq j,$ let $\s(i,j)=s_js_{j-1}\dots s_i$, interpreting $\s(i,i)$ as $s_i$.  Note that $a_j=\s(i,j)a_{i-1}$.  Now, suppose that there are infinitely many $i\in\N$ such that $\s(1,i)\in S{\sm}T.$  Then, since there are only finitely many $\l$-classes of ${_T}S_T$ contained in $S{\sm}T,$ there exist $i,j\in\N$ with $i<j$ such that $\s(1,i)\,\l\,\s(1,j)$ in ${_T}S_T$.  But then
$$a_i=\s(1,i)a_0\,\l\,\s(1,j)a_0=a_j$$ 
in ${_T}S_T$, and hence in $S,$ a contradiction.  Thus, there exists $i_1\in\N$ such that $\s(1,i_1)\in T.$  Now consider the elements $\s(i_1+1,i)$ ($i\geq i_1+1$).  By essentially the same argument as above, there exists $i_2\geq i_1+1$ such that $\s(i_1+1,i_2)\in T.$  Continuing in this way, we obtain an infinite chain $0=i_0<i_1<i_2<\cdots$ of non-negative integers such that $\s(i_k+1,i_{k+1})\in T$ for all $k\geq 0.$  It follows that $a_{i_1}\geql a_{i_2}\geql\cdots$ in ${_T}S_T$.  Since ${_T}S_T$ satisfies $M_L,$ there exists $m\in\N$ such that $a_{i_m}\,\l\,a_{i_n}$ in ${_T}S_T$ for all $n\geq m.$  But then $a_{i_m}\,\l\,a_{i_n}$ in $S$ for all $n\geq m,$ a contradiction.  This completes the proof.
\end{proof}

\begin{cor}\label{cor:Greenindex}
Let $S$ be a semigroup, and let $T$ be a subsemigroup of $S$ of finite Green index.  Then the following are equivalent:
\begin{enumerate}[leftmargin=*]
\item[\emph{(1)}] $S$ satisfies $M_L$;
\item[\emph{(2)}] $T$ satisfies $M_L$;
\item[\emph{(3)}] ${_T}S_T$ satisfies $M_L$.
\end{enumerate}
\end{cor}


In the statement of Theorem \ref{thm:subsemigroup}, we cannot replace the assumption that ${_T}(S/T)_T$ has finitely many $\l$-classes with the weaker condition that ${_T}(S/T)_T$ satisfies $M_L$, as demonstrated by the following example.

\begin{ex}\label{ex:subsemigroup}
Consider the group of integers $\Z$ and its subsemigroup $\N.$  Let $\k\in\{\l,\r,\j\}.$  Clearly $\Z$ has a single $\k$-class, so certainly satisfies $M_K$.  On the other hand, in ${_\N}\Z_{\N}$ we have
$$\cdots\gk-2\gk-1\gk0\gk1\gk2\gk\cdots,$$
and hence ${_\N}\Z_{\N}$ does not satisfy $M_K$ (and neither does $\N$).  Moreover, it is clear that ${_\N}(\Z/\N)_{\N}$ satisfies $M_K$.
\end{ex}

\begin{op} 
Do the analogues of Theorem \ref{thm:subsemigroup} and Corollary \ref{cor:Greenindex} for the property $M_J$ hold?
\end{op}

Let $T$ be a subsemigroup of a semigroup $S.$  For $\k\in\{\l,\r,\j\},$ we say that $T$ is {\em$\k$-preserving} (in $S$) if the $\k$-preorder on $T$ coincides with the restriction to $T\times T$ of the $\k$-preorder on $S.$  It is easy to show that if $S{\sm}T$ is a right/left ideal of $S$ then $T$ is $\l$/$\r$-preserving, and if $S{\sm}T$ is a two-sided ideal then $T$ is also $\j$-preserving.  Also, if $T$ is {\em regular}, i.e.\ if $a\in aTa$ for all $a\in T,$ then $T$ is $\l$- and $\r$-preserving \cite[Lemma 3]{East}.  

The following lemma is clear, and the proof is omitted.

\begin{lem}\label{lem:preserving}
Let $\k\in\{\l,\r,\j\},$ and let $M_K$ be the minimal condition associated with $\k.$  Let $T$ be a $\k$-preserving subsemigroup of a semigroup $S.$  If $S$ satisfies $M_K$ then so does $T.$
\end{lem}

\begin{cor}\label{cor:preserving}
Let $T$ be a subsemigroup of a semigroup $S.$
\begin{enumerate}[leftmargin=*]
\item[\emph{(1)}] Suppose that $T$ is regular.  If $S$ satisfies $M_L$ then so does $T.$
\item[\emph{(2)}] Suppose that $S{\sm}T$ is a right ideal of $S.$  If $S$ satisfies $M_L$ then so does $T.$
\item[\emph{(3)}] Suppose that $S{\sm}T$ is an ideal of $S.$  If $S$ satisfies $M_J$ then so does $T.$
\end{enumerate}
\end{cor}

\begin{rem}
A regular subsemigroup of a semigroup that satisfies $M_J$ need not satisfy $M_J$ itself.  Indeed, any semigroup (and hence any regular semigroup that does not satisfy $M_J$) may be embedded in a bisimple monoid \cite[Corollary 1]{Higgins:1990}.
\end{rem}

A {\em bi-ideal} of a semigroup $S$ is a subset $B$ of $S$ such that $BS^1B\subseteq B.$  The notion of bi-ideals, introduced in \cite{Good}, generalises that of one-sided (and hence two-sided) ideals.

\begin{thm}\label{thm:bi-ideal}
Let $B$ be a bi-ideal of a semigroup $S.$  If $S$ satisfies $M_L$ then so does $B.$
\end{thm}

\begin{proof}
Aiming for a contradiction, suppose that there is an infinite descending chain $a_0\gl a_1\gl\cdots$ in $B.$  Then $a_i\in a_{i-1}B$ for all $i\in\N,$ and $a_0\geql a_1\geql\cdots$ in $S.$  Since $S$ satisfies $M_L$, there exists $m\in\N$ such that $a_m\,\l\,a_n$ in $S$ for all $n\geq m.$  It follows that
$$a_{m+1}\in a_mB\subseteq a_{m+3}S^1B\subseteq a_{m+2}BS^1B\subseteq a_{m+2}B,$$
where for the final inclusion we use the fact that $B$ is a bi-ideal of $S.$  But this contradicts that $a_{m+1}\gl a_{m+2}$.  Thus $B$ satisfies $M_L$.
\end{proof}

By Corollary \ref{cor:ideal} below, the property of satisfying $M_J$ is not in general inherited by bi-ideals, or even by ideals.  However, we have:

\begin{cor}
Let $B$ be a bi-ideal of a semigroup $S.$  If $S$ is stable and satisfies $M_J$, then $B$ satisfies $M_J$.
\end{cor}

\begin{proof}
Since $S$ is stable and satisfies $M_J$, by Proposition \ref{prop:M_L+M_R} it satisfies $M_L$ and $M_R$.  It then follows from Theorem \ref{thm:bi-ideal} that $B$ satisfies $M_L$ and $M_R$, and hence $B$ satisfies $M_J$ by Proposition \ref{prop:M_L+M_R}.
\end{proof} 

Now, combining Proposition \ref{prop:actex} and Lemma \ref{lem:quotient}, we have:

\begin{prop}\label{prop:sgpext}
Let $S$ be a semigroup and $I$ an ideal of $S.$  For $K\in\{L,R,J\},$ the semigroup $S$ satisfies $M_K$ if and only if ${_S}I_S$ and $S/I$ satisfy $M_K$.
\end{prop}

In fact, a semigroup satisfies $M_L$ (or $M_R$) if and only if both an ideal and the associated Rees quotient do.

\begin{thm}\label{thm:sgpext}
Let $S$ be a semigroup and $I$ an ideal of $S.$  Then $S$ satisfies $M_L$ if and only if both $I$ and $S/I$ satisfy $M_L$.
\end{thm}

\begin{proof}
The forward implication follows from Theorem \ref{thm:bi-ideal} and Proposition \ref{prop:sgpext}.  For the converse, assume for a contradiction that there exists an infinite strictly descending chain $a_0\gl a_1\gl a_2\gl\cdots$ in $S.$  Then, for each $i\in\N,$ there exists $s_i\in S$ such that $a_i=s_ia_{i-1}$.  For each pair $i,j\in\N$ with $i\leq j,$ let $\s(i,j)=s_js_{j-1}\dots s_i$, interpreting $\s(i,i)$ as $s_i$.  Note that $a_j=\s(i,j)a_{i-1}$.  Now, clearly 
$$\s(1,1)\geql\s(1,2)\geql\s(1,3)\geql\cdots\quad\text{in }S.$$ 
Suppose that $\s(1,i)\in S{\sm}I$ for all $i\in\N.$  Then we have 
$$\s(1,1)\geql\s(1,2)\geql\s(1,3)\geql\cdots\quad\text{in }S/I.$$ 
Since $S/I$ satisfies $M_L$, there exists $m\in\N$ such that $\s(1,m)\,\l\,\s(1,n)$ (in $S/I$) for all $n\geq m,$ which implies that $\s(1,m)\,\l\,\s(1,n)$ in $S$ for all $n\geq m.$  But then $a_m=\s(1,m)a_0\,\l\,\s(1,n)a_0=a_n$ (in $S$) for all $n\geq m,$ a contradiction.  Thus, there exists $i_1\in\N$ such that $\s(1,i_1)\in I.$  Now consider the chain 
$$\s(i_1+1,i_1+1)\geql\s(i_1+1,i_1+2)\geql\s(i_1+1,i_1+2)\geql\cdots\quad\text{in }S.$$ 
By essentially the same argument as above, there exists $i_2\geq i_1+1$ such that $\s(i_1+1,i_2)\in I.$  Continuing in this way, we obtain an infinite chain $0=i_0<i_1<i_2<\cdots$ of non-negative integers such that $\s(i_k+1,i_{k+1})\in I$ for all $k\geq 0.$  Thus, for each $k\geq0$ we have $a_{i_{k+1}}=\s(i_k+1,i_{k+1})a_{i_k}\in Ia_{i_k}\subseteq I.$  It follows that $a_{i_1}\geql a_{i_2}\geql\cdots$ in $I.$  Since $I$ satisfies $M_L,$ there exists $N\in\N$ such that $a_{i_N}\,\l\,a_{i_n}$ in $I$ for all $n\geq N.$  But then $a_{i_N}\,\l\,a_{i_n}$ in $S$ for all $n\geq N,$ a contradiction.  This completes the proof.
\end{proof}

We conclude this section by showing that the analogue of Theorem \ref{thm:sgpext} for $M_J$ fails in both directions.  To this end, we introduce the following construction.

\begin{con}\label{con:S,T;A}
Let $S$ and $T$ be disjoint semigroups, and let $A$ be an $(S,T)$-biact.  Let $\{x_a : a\in A\}\sqcup\{0\}$ be a set in one-to-one correspondence with $A\sqcup\{0\}$ and disjoint from $S\cup T,$ and let $U=S\cup T\cup\{x_a : a\in A\}\cup\{0\}.$  Define a multiplication on $U,$ extending those of $S$ and $T,$ as follows:
$$sx_a=x_{sa},\quad x_at=x_{at},\quad st=ts=x_as=tx_a=x_ax_b=0u=u0=0$$
for all $s\in S,$ $t\in T,$ $a,b\in A$ and $u\in U.$  It is straightforward to show that $U$ is a semigroup under this multiplication, and we denote it by $\U(S,T;A).$ 
\end{con}

\begin{prop}\label{prop:S,T;A}
Let $S$ and $T$ be disjoint semigroups, let $A$ be an $(S,T)$-biact, and let $U=\U(S,T;A).$  Let $N=\{x_a : a\in A\}\cup\{0\},$ and let $I=S\cup N.$
\begin{enumerate}[leftmargin=*]
\item The set $I$ is an ideal of $U,$ and $N$ is an ideal of $I$ (and hence of $U$).  Moreover, $N$ is a null semigroup (i.e.\ $N^2=\{0\}$), and hence satisfies $M_J$.
\item The semigroup $U$ satisfies $M_J$ if and only if $S,$ $T$ and $A$ satisfy $M_J$.
\item The semigroup $I$ satisfies $M_J$ if and only if $S$ satisfies $M_J$ and $A$ satisfies $M_L$.
\end{enumerate}
\end{prop}

\begin{proof}
(1) This follows immediately from the definition of the multiplication in $U.$

(2)  Clearly the Rees quotient $U/N$ is isomorphic to the subsemigroup $S\cup T\cup\{0\}$ of $U,$ and the latter semigroup is the {\em 0-direct union} of $S$ and $T$; i.e.\ $S\cap T=\emptyset,$ $0\notin S\cup T$ and $ST=TS=\{0\}.$  It follows from \cite[Lemma 1.4]{Hall} that $U/N$ satisfies $M_J$ if and only if $S$ and $T$ satisfy $M_J$.  Moreover, it is straightfoward to show that, for any $a,b\in A,$ we have $x_a\leqj x_b$ in $U$ if and only if $a\leqj b$ in $A.$  Consequently, the $U$-biact ${_U}N_U$ satisfies $M_J$ if and only if $A$ satisfies $M_J$.  The statement of (2) now follows from Proposition \ref{prop:sgpext}.

(3) Clearly $I/N\cong S^0$ satisfies $M_J$ if and only if $S$ satisfies $M_J$.  Moreover, it is straightfoward to show that, for any $a,b\in A,$ we have $x_a\leqj x_b$ in $I$ if and only if $a\leql b$ in $A,$ and hence ${_I}N_I$ satisfies $M_J$ if and only if $A$ satisfies $M_L$.  The statement of (3) now follows from Proposition \ref{prop:sgpext}. 
\end{proof}

There certainly exists a pair of disjoint semigroups $S$ and $T$ with an $(S,T)$-biact $A$ such that $S,$ $T$ and $A$ satisfy $M_J$ but $A$ does not satisfy $M_L$.  For instance, we can take $S$ to be the bicyclic monoid and $T=\{\ol{s} : s\in S\}$ a disjoint copy of $S,$ and define $A=\{a_s : s\in S\}$ with actions given by $sa_t=a_s\ol{t}=a_{st}$ for all $s,t\in S.$  Thus, by Proposition \ref{prop:S,T;A}, we have:

\begin{cor}\label{cor:ideal}
There exists a semigroup $U$ with an ideal $I,$ and an ideal $N$ of $I,$ such that $U,$ $N$ and $I/N$ satisfy $M_J$ but $I$ does not satisfy $M_J$.
\end{cor}

\section{Stability}\label{sec:stability}

Unlike the conditions $M_L$, $M_R$ and $M_J$, stability for semigroups/biacts is not closed under quotients.  Indeed, free semigroups are certainly stable since they are $\j$-{\em trivial} (i.e.\ $a\,\j\,b\Leftrightarrow a=b$), and of course every semigroup is a quotient of a free semigroup.  To see that stability is not closed under biact quotients, consider any free semigroup $S$ with a congruence $\rho$ such that $S/{\rho}$ is not stable.  Then the $S$-biact ${_S}S_{S}$ is stable (since $S$ is stable), and it follows from the first part of Lemma \ref{lem:quotient} that ${_S}S_S/\rho$ is not stable.


On the other hand, we have the following analogue of Proposition \ref{prop:actex}.

\begin{prop}\label{prop:actex,stab}
Let $S$ and $T$ be semigroups, let $A$ be an $(S,T)$-biact, and let $B$ a subact of $A.$  Then $A$ is (left) stable if and only if both $B$ and $A/B$ are (left) stable.
\end{prop}

\begin{proof}
For $\k\in\{\l,\r\},$ it is easy to see that for any $a,b\in A$ with $a\,\j\,b$ in $A,$ we have $a\leqk b$ in $A$ if and only if $a\leqk b$ in $B$ or $a\leqk b$ in $A/B.$  The result quickly follows.
\end{proof}

From now on we focus on semigroups.  Recall that the Green index of a subsemigroup $T$ in a semigroup $S$ is the number of $\h$-classes of the Rees quotient ${_T}(S/T)_T$ of the $T$-biact ${_T}S_T$ by its subact ${_T}T_T$.  As ${_T}T_T$ is a subact of ${_T}S_T$, by Proposition \ref{prop:actex,stab} we have:

\begin{prop}\label{prop:subsemigroup,stab}
Let $T$ be a subsemigroup of a semigroup $S.$  Then ${_T}S_T$ is (left) stable if and only $T$ and ${_T}(S/T)_T$ are (left) stable.
\end{prop}

We shall establish an analogue of Corollary \ref{cor:Greenindex} for stability.  First, we have:

\begin{prop}\label{prop:subsemigroup,L-classes}
Let $S$ be a semigroup, and let $T$ be a subsemigroup of $S$ such that ${_T}(S/T)_T$ has finitely many $\l$-classes.
\begin{enumerate}[leftmargin=*]
\item[\emph{(1)}] If $S$ is left stable then $T$ is left stable.
\item[\emph{(2)}] The semigroup $T$ is left stable if and only if ${_T}S_T$ is left stable.
\end{enumerate}
\end{prop}

\begin{proof}
Clearly ${_T}(S/T)_T$ satisfies $M_L$, and is hence left stable by Lemma \ref{lem:L-periodic}, so (2) follows from Proposition \ref{prop:subsemigroup,stab}.

For (1), let $a,b\in T$ with $a\leql b$ and $a\,\j\,b$ in $T.$  Then there exist $t,u,v\in T^1$ such that $a=tb$ and $b=uav.$  It follows that $a=(tu)^iav^i$ and $b=u(tu)^{i-1}av^i$ for all $i\in\N.$  Let $n$ be the number of $\l$-classes of ${_T}(S/T)_T$.  Clearly $(tu)^na\leql a\,\j\,(tu)^na$ in $T$ (and hence in $S$).  Since $S$ is left stable, it follows that $(tu)^na\,\l\,a,$ so there exists $s\in S^1$ such that $a=s(tu)^na.$  For $i\in\{1,\dots,n\},$ let $s_i=u(tu)^{i-1}s.$  We then have
$$b=u(tu)^{i-1}av^i=u(tu)^{i-1}s(tu)^nav^i=s_i(tu)^nav^i=s_i(tu)^{n-i}(tu)^iav^i=s_i(tu)^{n-i}a.$$
Thus, if $s_i\in T^1$ for some $i\in\{1,\dots,n\},$ then $b\in T^1a$ and hence $a\,\l\,b$ in $T,$ and we are done.  Suppose then that $s_i\in S{\sm}T$ for all $i\in\{1,\dots,n\}.$  Since there are only $n-1$ $\l$-classes of ${_T}S_T$ in $S{\sm}T,$ by the pigeonhole principle there exist $i,j\in\{1,\dots,n\}$ with $i<j$ such that $s_i\,\l\,s_j$.  Thus, there exists $x\in T^1$ such that $s_i=xs_j$.  It follows that
\begin{align*}
b&=s_i(tu)^nav^i=xs_j(tu)^nav^i=xu(tu)^{j-1}s(tu)^nav^i=xu(tu)^{j-1}av^i=xu(tu)^{j-1-i}(tu)^iav^i\\
&=xu(tu)^{j-1-i}a\in T^1a.
\end{align*}
Thus $a\,\l\,b$ in $T,$ as required.
\end{proof}

\begin{thm}\label{thm:Greenindex,stab}
Let $S$ be a semigroup, and let $T$ be a subsemigroup of $S$ of finite Green index.  Then the following are equivalent:
\begin{enumerate}[leftmargin=*]
\item[\emph{(1)}] $S$ is (left) stable;
\item[\emph{(2)}] $T$ is (left) stable;
\item[\emph{(3)}] ${_T}S_T$ is (left) stable.
\end{enumerate}
\end{thm}

\begin{proof}
We prove the result for left stability.  Then, by symmetry, a dual statement holds for right stability, and the result for (two-sided) stability follows.   

Given Proposition \ref{prop:subsemigroup,L-classes}, we only need to prove that (2) implies (1).  So, let $a,b\in S$ with $a\leql b$ and $a\,\j\,b$ in $S.$  Then there exist $s,u,v\in S^1$ such that $a=sb$ and $b=uav.$  Then $a=(su)^iav^i$ and $b=u(su)^iav^{i+1}$ for all $i\in\N.$ 

Suppose first that there are infinitely many $n\in\N$ such that $(su)^n\in S{\sm}T$ or $(su)^na\in S{\sm}T.$  Since $T$ has finite Green index in $S,$ and since $\h\subseteq\l$ and $\l$ is right compatible, it follows that there exist $i,j\in\N$ with $i<j$ such that $(su)^ia\,\l\,(su)^ja$ in ${_T}(S/T)_T$.  Thus, there exists $x\in T^1$ such that $(su)^ia=x(su)^ja.$  It follows that
$$b=u(su)^iav^{i+1}=ux(su)^jav^{i+1}=ux(su)^{j-i-1}(su)^{i+1}av^{i+1}=ux(su)^{j-i-1}a,$$
and hence $a\,\l\,b$ in $S.$

Now suppose that there are only finitely many $n\in\N$ such that $(su)^n\in S{\sm}T$ or $(su)^na\in S{\sm}T.$  Then there exists $m\in\N$ such that $(su)^n,(su)^na\in T$ for all $n\geq m.$  Thus, for each $n\geq m,$ we have $(su)^{m+n}a=(su)^n(su)^ma\in T^1(su)^ma.$  We claim that there exists $N\geq m$ such that $(su)^{m+N}a\,\j\,(su)^ma$ in $T.$  Indeed, for each $n\geq m$ we have 
$$(su)^ma=(su)^m(su)^{m+n}av^{m+n}\in T^1(su)^{m+n}av^{m+n}.$$  
Thus, if $v^{m+n}\in T^1$ for some $n\geq m,$ the claim holds with $N=n.$  Suppose then that $v^{m+n}\in S{\sm}T$ for all $n\geq m.$  Since $T$ has finite Green index in $S,$ it follows that there exist $i,j\geq m$ with $j-i\geq m$ such that $v^i\,\h\,v^j$.  Thus, there exists $y\in T^1$ such that $v^j=v^iy.$  It follows that
$$(su)^ma=(su)^{m+j}av^iy
=(su)^{m+j-i}(su)^iav^iy=(su)^{m+j-i}ay\in(su)^{m+j-i}aT^1,$$
and hence the claim holds with $N=j-i.$  This establishes the claim.  Now, since $T$ is left stable, there exists $t\in T^1$ such that $(su)^ma=t(su)^{m+N}a.$  We then have
$$b=u(su)^mav^{m+1}=ut(su)^{m+N}av^{m+1}
=ut(su)^{N-1}(su)^{m+1}av^{m+1}=ut(su)^{N-1}a\in S^1a.$$
We conclude that $a\,\l\,b$ in $S,$ as required.
\end{proof}


The next result is an analogue of Lemma \ref{lem:preserving} for stability.  As with that lemma, the proof is omitted.

\begin{lem}
Let $T$ be a subsemigroup of a semigroup $S.$  If $S$ is left stable and $T$ is $\l$-preserving in $S,$ then $T$ is left stable.  Moreover, if $S$ is stable and $T$ is both $\l$- and $\r$-preserving in $S,$ then $T$ is stable.
\end{lem}

We have already noted that regular subsemigroups and subsemigroups whose complements are ideals are $\l$- and $\r$-preserving.  Recall that a {\em retract} of a semigroup $S$ is a subsemigroup $T$ of $S$ such that there is a homomorphism $\th : S\to T$ with $t\th=t$ for all $t\in T.$  It is easy to show that retracts are $\l$- and $\r$-preserving.  Thus, we have:

\begin{cor}
Let $T$ be a subsemigroup of a semigroup $S$ such that (at least) one of the following holds: $T$ is regular; $T$ is a retract of $S$; $S{\sm}T$ is an ideal of $S.$  If $S$ is (left) stable, then $T$ is (left) stable.
\end{cor}

Stability is also inherited by bi-ideals:

\begin{thm}
Let $B$ be a bi-ideal of a semigroup $S.$  If $S$ is (left) stable, then $B$ is (left) stable.
\end{thm}

\begin{proof}
It suffices to prove that if $S$ is left stable then so is $B.$  So, let $a,b\in B$ with $a\leql b$ and $a\,\j\,b$ in $B.$  Then there exist $u,x,y\in B^1$ such that $a=ub$ and $b=xay.$  If $u=1$ then $a=b,$ and we are done.  Suppose now that $u\neq 1,$ and let $z=xux.$  Then $z\in B.$  We have
$$b=xay=xuby=xuxay^2=zay^2.$$
It follows that 
$$za=zub=zuzay^2\quad\text{ and }\quad ay^2=uby^2=uzay^2y^2,$$
so that $za\leql ay^2\,\j\,za$ in $B$ (and hence in $S$).  Since $S$ is left stable, there exists $s\in S^1$ such that $ay^2=sza.$  It follows that $b=zsza\in(BS^1B)a\subseteq Ba.$  We conclude that $a\,\l\,b$ in $B,$ as required.
\end{proof}

The remainder of this section is concerned with Rees quotients and ideal extensions.  The following lemma is easy to prove.

\begin{lem}\label{lem:RQ,stab}
Let $S$ be a semigroup and $I$ an ideal of $S.$  Then the semigroup $S/I$ is (left) stable if and only if the $S$-biact ${_S}S_S/{_S}I_S$ is (left) stable.
\end{lem}

Combining Proposition \ref{prop:actex,stab} and Lemma \ref{lem:RQ,stab}, we obtain:

\begin{prop}\label{prop:sgpex,stab}
Let $S$ be a semigroup and $I$ an ideal of $S.$  Then $S$ is (left) stable if and only if ${_S}I_S$ and $S/I$ are (left) stable.
\end{prop}

We conclude this section by showing that a semigroup may not be stable even if both an ideal and the associated Rees quotient are stable.  In order to do so, we introduce a construction similar to Construction \ref{con:S,T;A}.

\begin{con}\label{con}
Let $S$ be a semigroup and $A$ an $S$-biact.  Let $\{x_a : a\in A\}$ be a set in one-to-one correspondence with $A$ and disjoint from $S,$ and let $0$ be an element disjoint from $S\cup\{x_a : a\in A\}.$  Define a multiplication on $U=S\cup\{x_a : a\in A\}\cup\{0\},$ extending that of $S,$ by
$$sx_a=x_{sa},\quad x_as=x_{as},\quad x_ax_b=u0=0u=0\quad\text{for all }\,s\in S,\,a,b\in A,\,u\in U.$$  
It is straightforward to show that $U$ is a semigroup under this multiplication, and we denote it by $\U(S,A).$  Notice that $\{x_a : a\in A\}\cup\{0\}$ is both an ideal of $\U(S,A)$ and a null semigroup.
\end{con}

\begin{prop}\label{prop:con}
Let $S$ be a semigroup, let $A$ be an $S$-biact, and let $U=\U(S,A).$  Then $U$ is (left) stable if and only if $A$ and $S$ are (left) stable.
\end{prop}

\begin{proof}
Let $I=\{x_a : a\in A\}\cup\{0\}.$  By Proposition \ref{prop:sgpex,stab}, the semigroup $U$ is (left) stable if and only if ${_U}I_U$ and $U/I$ are (left) stable.  It is straightforward to show that, for $\k\in\{\l,\r,\j\},$ we have $x_a\leqk x_b$ in $U$ if and only if $a\leqk b$ in $S,$ and hence ${_U}I_U$ is (left) stable if and only if $A$ is (left) stable.  Finally, $U/I\cong S^0$ is (left) stable if and only if $S$ is (left) stable.
\end{proof}

\begin{cor}
There exists a semigroup $U$ with an ideal $I$ such that both $I$ and $U/I$ are stable but $U$ is not stable.
\end{cor}

\begin{proof}
Take a stable semigroup $S$ with an $S$-biact $A$ that is not stable.  (The existence of such a semigroup and biact was established at the beginning of this section.)
By Proposition \ref{prop:con}, the semigroup $U=\U(S,A)$ is not stable.  On the other hand, the ideal $I=\{x_a : a\in A\}\cup\{0\}$ is $\j$-trivial and hence stable, and the Rees quotient $U/I\cong S^0$ is stable since $S$ is stable.
\end{proof}

\vspace{0.5em}
\section*{Acknowledgements}
This work was supported by the Engineering and Physical Sciences Research Council [EP/V002953/1].  The author thanks the referee for helpful comments.


\vspace{0.5em}
\begin{thebibliography}{99}
\bibitem{Arbib}
\textit{Algebraic Theory of Machines, Languages and Semigroups}.  Edited by M.\! Arbib, with a major contribution by K.\! Krohn and L.\! Rhodes.  Academic Press, 1968.
\bibitem{Cain}
A.\! Cain, R.\! Gray, N.\! Ru{\v s}kuc.  Green index in semigroups: generators, presentations, and automatic structures.  \textit{Semigroup Forum}, 85:448--476, 2012.
\bibitem{Clifford:1961}
A.\! Clifford, G.\! Preston.  {\em The Algebraic Theory of Semigroups: Volume I}.   Amer. Math. Soc., 1961.
\bibitem{Clifford:1967}
A.\! Clifford, G.\! Preston.  {\em The Algebraic Theory of Semigroups: Volume II}.   Amer. Math. Soc., 1967.
\bibitem{Drazin}
M.\! Drazin.  Pseudo-inverses in associative rings and semigroups.  \textit{Amer. Math. Mon.}, 65:506--514, 1958.
\bibitem{East}
J.\! East, P.M.\! Higgins.  Green's relations and stability for subsemigroups.  \textit{Semigroup Forum}, 101:77--86, 2020.
\bibitem{Good}
R.\! Good, D.\! Hughes.  Associated groups for a semigroup.  \textit{Bull. Amer. Math. Soc.}, 58:624--625, 1952.
\bibitem{Hall}  
T.\! Hall, W.\! Munn.  Semigroups satisfying minimal conditions II.  \textit{Glasgow Math. J.}, 20:133--140, 1979.
\bibitem{Higgins:1990} 
P. Higgins.  Embedding in bisimple semigroups.  \textit{Semigroup Forum}, 40:105--107, 1990.
\bibitem{Higgins:1992}
P.M.\! Higgins.  \textit{Techniques of Semigroup Theory}.  Oxford University Press, 1992.
\bibitem{Gray}
R.\! Gray, N.\! Ru{\v s}kuc.  Green index and finiteness conditions for semigroups.  \textit{J. Algebra}, 320:3145--3164, 2008.
\bibitem{Green}
J.\! Green.  On the structure of semigroups.  \textit{Ann. Math.}, 54:163--172, 1951.
\bibitem{kkm}
M.\! Kilp, U.\! Knauer, A.\! Mikhalev.  \textit{Monoids, Acts and Categories}.  Walter de Gruyter, 2000.
\bibitem{Lallement}
G.\! Lallement.  \textit{Semigroups and Combinatorial Applications}.  Wiley, 1979.
\bibitem{Mary}
X.\! Mary.  A local structure theorem for stable, $\j$-simple semigroup biacts.  \textit{Semigroup Forum}, 99:724--753, 2019.
\bibitem{Munn}  
W.\! Munn.  Semigroups satisfying minimal conditions.  \textit{Proc. Glasgow Math. Assoc.}, 3:145--152, 1957.
\bibitem{Rhodes}
J.\! Rhodes, B.\! Steinberg.  \textit{The $q$-theory of Finite Semigroups}.  Springer, 2009.
\bibitem{Wallace}
A.\! Wallace.  Relative ideals in semigroups.  II (The relations of Green).  \textit{Acta Math. Acad. Sci. Hungarica}, 14:137--148, 1963.
\end{thebibliography}
\end{document}